\documentclass[12pt]{amsart}
\usepackage{amssymb,amsmath,amsthm,latexsym}
\usepackage{mathrsfs}
\usepackage{mdwlist}
\usepackage{graphicx}
\usepackage{a4wide}
\input xy
\xyoption{all}

\newcommand{\N}{\mathbb{N}}

\newcommand{\C}{\mathbb{C}}
\newcommand{\K}{\mathbb{K}}

\newcommand{\p}{\varphi}
\newcommand{\e}{\varepsilon}
\newcommand{\w}{\widetilde}
\newcommand{\oo}{\overline}
\newcommand{\om}{\omega}

\newcommand{\ind}{\boldsymbol{1}}
\newcommand{\n}[1]{\|#1\|}
\newcommand{\wsc}[1]{\overline{#1}^{w\ast}} 
\renewcommand{\span}{\mathrm{span}}

\newcommand{\ccup}{\scalebox{0.85}{$\bigcup$}}

\newcommand{\B}{\mathscr{B}}

\newcommand{\J}{\mathcal{J}}

\newcommand{\W}{\mathscr{W}}

\newcommand{\WCG}{\mathscr{W\!C\!G}}

\newtheorem{theorem}{Theorem}[section]
\newtheorem{lemma}[theorem]{Lemma}

\newtheorem{proposition}[theorem]{Proposition}
\newtheorem{corollary}[theorem]{Corollary}

\theoremstyle{definition}

\newtheorem{example}[theorem]{Example}
\newtheorem{remark}[theorem]{Remark}

\theoremstyle{remark}

\numberwithin{equation}{section}

\newcommand{\abs}[1]{\lvert#1\rvert}

\newcounter{smallromans}

\newenvironment{romanenumerate}
{\begin{list}{{\normalfont\textrm{(\roman{smallromans})}}}%
    {\usecounter{smallromans}\setlength{\itemindent}{0cm}%
      \setlength{\leftmargin}{5.5ex}\setlength{\labelwidth}{5.5ex}%
      \setlength{\topsep}{0.75\parsep}\setlength{\partopsep}{0ex}%
      \setlength{\itemsep}{0ex}}}%
  {\end{list}}

\newcounter{smallalphs}

\newenvironment{alphenumerate}
{\begin{list}{{\normalfont\textrm{(\alph{smallalphs})}}}%
    {\usecounter{smallalphs}\setlength{\itemindent}{0cm}%
      \setlength{\leftmargin}{5.5ex}\setlength{\labelwidth}{5.5ex}%
      \setlength{\topsep}{0.75\parsep}\setlength{\partopsep}{0ex}%
      \setlength{\itemsep}{0ex}}}%
  {\end{list}}

\begin{document}
\title[Weakly compactly generated operators acting on a Banach space]{The ideal of weakly compactly generated operators acting on a Banach space}

\author[T. Kania]{Tomasz Kania}
\address{Department of Mathematics and Statistics, Fylde College,
  Lancaster University, Lancaster LA1 4YF, United Kingdom and Institute of Mathematics, Polish Academy of Sciences, \'{S}niadeckich~8, 00-956 Warszawa, Poland}

\email{t.kania@lancaster.ac.uk}

\author[T. Kochanek]{Tomasz Kochanek}
\address{Institute of Mathematics, Polish Academy of Sciences, \'{S}niadeckich~8, 00-956 Warszawa, Poland and Institute of Mathematics, University of Silesia, Bankowa 14, 40-007 Katowice, Poland}
\email{t\_kochanek@wp.pl}


\hyphenation{ope-ra-tor  ope-ra-tors  auto-ma-ti-cally}

\begin{abstract}
We call a bounded linear operator acting between Banach spaces \emph{weakly compactly generated} ($\mathsf{WCG}$ for short) if its range is contained in a~weakly compactly generated subspace of its target space. This notion simultaneously generalises being weakly compact and having separable range. In a comprehensive study of the class of $\mathsf{WCG}$ operators, we prove that it forms a~closed surjective operator ideal and investigate its relations to other classical operator ideals. By considering the $p$th long James space $\J_p(\om_1)$, we show how properties of the ideal of $\mathsf{WCG}$ operators (such as being the unique maximal ideal) may be used to derive results outside ideal theory. For instance, we identify the $K_0$-group of $\B(\J_p(\om_1))$ as the additive group of integers.
\end{abstract}

\maketitle

\section{Introduction}
\noindent
Amir and Lindenstrauss \cite{AL} initiated the study of \emph{weakly compactly generated} ($\mathsf{WCG}$ for short) Banach spaces, that is, Banach spaces containing a~weakly compact fundamental subset. Any reflexive and any separable Banach space is weakly compactly generated. Other notable examples include $L_1(\mu)$-spaces with a~$\sigma$-finite positive measure $\mu$, and $c_0(\Gamma)$-spaces for an~arbitrary index set $\Gamma$. The latter play a~special role, as for every $\mathsf{WCG}$ space $X$ there is a~bounded linear operator which maps $X$ injectively into $c_0(\Gamma)$ for some $\Gamma$. On the other hand, there are plenty of Banach spaces which are not $\mathsf{WCG}$ such as $\ell_\infty$ and $\ell_1(\Gamma)$ for any uncountable index set $\Gamma$.

According to Lindenstrauss \cite{LI}, the class of $\mathsf{WCG}$ Banach spaces is stable under quotients, $c_0$-sums, $\ell_p$-sums for $p\in(1,\infty)$, and countable $\ell_1$-sums. Surprisingly, a~closed subspace of a~$\mathsf{WCG}$ Banach space need not be $\mathsf{WCG}$. The first counterexample was given by Rosenthal \cite{Ro} who exhibited a~non-$\mathsf{WCG}$ subspace of $L_1(\mu)$ for some probability measure $\mu$. Note that the aforementioned spaces $\ell_\infty$ and $\ell_1(\Gamma)$ (with $\Gamma$ uncountable) are not subspaces of any $\mathsf{WCG}$ space. We refer to \cite{AM} for further examples concerning that subspace problem, and to \cite{hajek} for a~list of necessary and sufficient conditions for being a~(subspace of a) $\mathsf{WCG}$ Banach space.

Let $T\colon X\to Y$ be a bounded linear operator acting between Banach spaces $X$ and $Y$. We call the operator $T$ \emph{weakly compactly generated} (or $\mathsf{WCG})$ if there is a $\mathsf{WCG}$ subspace $Z$ of $Y$ such that $T(X)\subseteq Z$. We shall prove in Section 2 (Theorem \ref{mainwcg}) that the class $\mathscr{W\!C\!G}$ of all weakly compactly generated operators forms a closed operator ideal. Moreover, this operator ideal is a surjective, but neither it is injective nor symmetric (Propositions \ref{surjective} and \ref{non-symmetric}). We then compare $\mathscr{W\!C\!G}$ to other classical operator ideals, including operator ideals of weakly compact, completely continuous, strictly singular and strictly cosingular operators.

Section 3 is devoted to weakly compactly generated operators acting on the $p$th long James space $\J_p(\om_1)\;(p\in (1,\infty))$. The main result of this section (Theorem \ref{long_James}) asserts that the ideal of weakly compactly generated operators is the unique maximal ideal of the algebra $\mathscr{B}(\J_p(\om_1))$ of bounded operators on $\J_p(\om_1)$. Buliding on the techniques from \cite{kaniakoszmiderlaustsen}, some further descriptions of this ideal are given, and these lead to additional results concerning commutators, automatic continuity of homomorphisms and the $K_0$-group of $\B(\J_p(\om_1))$.

Next we turn our attention to operators acting on $C(K)$-spaces. In terms of the representing measure of a~given operator $T\colon C(K)\to X$, we give a~sufficient condition for $T$ being $\mathsf{WCG}$ (Theorem \ref{FMZ_theorem}). This is an~application of the~characterisation of subspaces of $\mathsf{WCG}$ Banach spaces, obtained by Fabian, Montesinos and Zizler \cite{FMZ}.

Finally, in Section 5 we discuss some examples of non-Eberlein compacta $K$ for which the ideal of $\mathsf{WCG}$ operators on $C(K)$ is maximal. In particular we show that this is the case for a~certain Mr\'owka space $K$ constructed by Koszmider \cite{Ko}, and we give a~complete description of the lattice of closed ideals in $\mathscr{B}(C(K))$ (Theorem \ref{mrowka}).

Throughout this paper, Banach spaces are assumed to be over the field $\mathbb{K}=\mathbb{R}$ or $\mathbb{K}=\mathbb{C}$, unless the field is explicitly specified. By an \emph{operator} we understand a bounded linear operator acting between Banach spaces. An operator $T\colon E\to F$ is \emph{bounded below} if there exists a constant $\gamma>0$ such that $\|Tx\|\geqslant \gamma \|x\|$ for every $x\in E$, which means that $T$ is one-to-one and has closed range. The space $\mathscr{B}(E,F)$ of all operators $T\colon E\to F$ is a Banach space, when endowed with the operator norm and $\mathscr{B}(E,E)=\mathscr{B}(E)$ is a unital Banach algebra with multiplication being composition of operators.

Let $\mathscr{B}$ be the class of all operators acting between arbitrary Banach spaces. By an \emph{operator ideal} we understand a subclass $\mathscr{J}$ of $\mathscr{B}$, containing the identity operator on the one-dimensional Banach space and which assigns to each pair $(E,F)$ of Banach spaces a (not necessarily closed) linear subspace $\mathscr{J}(E, F)= \mathscr{B}(E,F)\cap \mathscr{J}$ such that for any Banach spaces $X,Y,E,F$ and for any operators $T\in \mathscr{B}(X,E), S\in \mathscr{J}(E,F)$ and $R\in \mathscr{B}(F, Y)$ we have $RST\in \mathscr{J}(X,Y)$. An operator ideal $\mathscr{J}$ is \emph{closed}, if the subspace $\mathscr{J}(E,F)$ is closed in $\mathscr{B}(E,F)$ for any pair $(E,F)$ of Banach spaces. We refer to \cite{Pi} for the general theory of operator ideals.  

The classes $\mathscr{K}, \mathscr{W}$ and $\mathscr{X}$ of compact operators, weakly compact operators and operators with separable range are standard examples of closed operator ideals, respectively. 

For any Banach space $X$ the class $\mathscr{G}_X\subseteq \mathscr{B}$ which assigns to each pair $(E, F)$ of Banach spaces the subspace
\[\mathscr{G}_X(E,F) = \mbox{span}\{ST\colon T\in \mathscr{B}(E,X), S\in \mathscr{B}(X,F)\}\subseteq \mathscr{B}(E,F)\]
is \emph{the ideal of operators factoring through} $X$. In the case where $X$ contains a complemented copy of its Cartesian square $X\oplus X$, the set $\{ST\colon T\in \mathscr{B}(E,X), S\in \mathscr{B}(X,F)\}$ is itself a linear subspace of $\mathscr{B}(E,F)$, whence the symbol `span' above can be suppressed.

\section{The operator ideal of weakly compactly generated operators and its relations to other operator ideals}

Recall that $\mathscr{W\!C\!G}(E, F)$ denotes the set of all operators $T\colon E\to F$ with $T(E)$ contained in a weakly compactly generated subspace of $F$. 

\begin{theorem}\label{mainwcg}The class $\mathscr{W\!C\!G}$ is a closed operator ideal.\end{theorem}
\begin{proof}Let $E$ and $F$ be Banach spaces. Fix two operators $T$ and $S$ in $\mathscr{W\!C\!G}(E,F)$. We deduce that $T+S$ belongs to $\mathscr{W\!C\!G}(E,F)$. Indeed, let $K_T$ and $K_S$ be two weakly compact subsets of $E$ such that $T(E)\subseteq \overline{\mbox{span}}~K_T$ and $S(E)\subseteq \overline{\mbox{span}}~K_S$. The union $K_T\cup K_S$ is weakly compact and $(T+S)(E)\subseteq  \overline{\mbox{span}}~(K_T\cup K_S)$, hence $T+S\in \mathscr{W\!C\!G}(E, F)$.

Let $X, Y, E, F$ be Banach spaces and let $T\in \mathscr{B}(X,E)$, $S\in \mathscr{W\!C\!G}(E, Y)$ and $R\in \mathscr{B}(Y, F)$. We note that both $ST$ and $RS$ are in the class $\mathscr{W\!C\!G}$. Indeed, since $S(T(X))\subseteq S(E)$ and $S(E)$ is a subspace of a $\mathsf{WCG}$ subspace of $Y$ we have $ST\in \mathscr{W\!C\!G}(X, Y)$. Now, let $K$ be a weakly compact subset of $Y$ such that $S(E)\subseteq \overline{\mbox{span}}~K$. Every operator is weak-to-weak continuous, thus the image $R(K)$ is weakly compact and $RS(E)\subseteq R(\overline{\mbox{span}}~K)\subseteq\overline{\mbox{span}}~R(K)$. Consequently, $RS\in\mathscr{W\!C\!G}(X, Y)$.

Finally, we shall prove that $\mathscr{W\!C\!G}$ is closed. Let $(T_n)_{n=1}^\infty\subseteq \mathscr{W\!C\!G}(E,F)$ be a norm-convergent sequence of operators with limit $T$, say. Let $F_n$ be a $\mathsf{WCG}$ subspace of $F$ such that $T_n(E)\subseteq F_n$ ($n\in \mathbb{N}$). Define $W$ to be the $\ell_1$-sum of $(F_n)_{n=1}^\infty$, which is again a $\mathsf{WCG}$ Banach space \cite[Proposition 2.4]{LI}. Furthermore, let $J\colon W\to F$ be the operator defined by $J(x_n)_{n=1}^\infty = \sum_{n=1}^\infty x_n$. Since $W$ is $\mathsf{WCG}$, the space $\overline{J(W)}$ is $\mathsf{WCG}$ as well, and we have $T(E)\subseteq \overline{\mbox{span}}\bigcup_{n\in\mathbb{N}}F_n\subseteq \overline{J(W)}$. \end{proof}

\begin{proposition}The ideal $\mathscr{W\!C\!G}(E)$ is a proper ideal of $\mathscr{B}(E)$ if and only if $E$ is not weakly compactly generated.\end{proposition}
\begin{proof}This follows by considering the range of the identity operator on $E$.\end{proof}
Let $E, E_0$ and $F$ be arbitrary Banach spaces. Recall that an operator ideal $\mathscr{J}$ is \emph{surjective} if for any surjective operator $Q\in \mathscr{B}(E, E_0)$ and each operator $T\in \mathscr{B}(E_0, F)$, we have $T\in \mathscr{J}(E_0, F)$ provided $TQ\in \mathscr{J}(E,F)$. An operator ideal $\mathscr{J}$ is \emph{injective} if for each closed subspace $F_0$ of $F$ and every operator $T\in \mathscr{B}(E,F_0)$ with $\iota T\in \mathscr{J}(E,F)$ we have $T\in \mathscr{J}(E, F_0)$ (here $\iota \colon F_0\to F$ denotes the canonical embedding).
\begin{proposition}\label{surjective}The operator ideal $\mathscr{W\!C\!G}$ is surjective, but not injective.\end{proposition}
\begin{proof}For the surjectivity of $\mathscr{W\!C\!G}$, suppose that $T\in \mathscr{B}(E,F)$ satisfies $TQ\in \mathscr{W\!C\!G}(E_0, E)$ for some Banach space $E_0$ and some surjection $Q\in \mathscr{B}(E_0, E)$. Since $Q$ is a surjection, the ranges of $TQ$ and $T$ are the same, hence $T\in \mathscr{W\!C\!G}(E,F)$.

We observe that $\mathscr{W\!C\!G}$ is not injective. Indeed, it follows from the existence of closed subspace of a $\mathsf{WCG}$ space which is not $\mathsf{WCG}$ \cite{Ro}.\end{proof}

\begin{proposition}\label{non-symmetric}The operator ideal $\mathscr{W\!C\!G}$ is not symmetric, that is, the adjoint of a weakly compactly generated operator need not be weakly compactly generated. 

Conversely, if the adjoint of an operator $T$ is weakly compactly generated, then again, $T$ need not be. \end{proposition}
\begin{proof}The lack of symmetry is clear -- the identity operator on $\ell_1$ is $\mathsf{WCG}$, while its adjoint $I_{\ell_1}^*=I_{\ell_\infty}$ is not.

Now, let $J\! L$ be the Johnson--Lindenstrauss space (consult \cite{JL} for its definition and properties), which is known not to be a $\mathsf{WCG}$ space but its dual $J\! L^*=\ell_1\oplus \ell_2(\mathfrak{c})$ clearly is. Take $T$ to be the identity on ${J\! L}$. Then, $T=I_{J\! L}$ is not $\mathsf{WCG}$, whereas $T^*=I_{J\! L^*}$ is.\end{proof}

Let us observe that each operator which factors through a $\mathsf{WCG}$ space is weakly compactly generated. In particular, one can deduce from this that weakly compact operators are weakly compactly generated. The Davis--Figiel--Johnson--Pe\l{}czy\'{n}ski theorem \cite[Theorem 6.2.15]{diestel_uhl} characterises weakly compact operators as precisely those which admit a factorisation through a reflexive space, hence each weakly compact operator is $\mathsf{WCG}$; of course such a heavy machinery is superfluous in this case as it can be seen directly. Trivially, operators with separable range are $\mathsf{WCG}$ as well.

In the remaining part of this section we shall study order relations between $\mathscr{W\!C\!G}$ and some classical operator ideals. Recall that an operator $T\in \mathscr{B}(E,F)$ is \begin{itemize}
\item[(a)]\emph{completely continuous} (or \emph{Dunford--Pettis}) if, it maps weakly convergent sequences in $E$ to norm convergent sequences in $F$;
\item[(b)]\emph{strictly singular} (or \emph{Kato}) if, it is not bounded below when restricted to any closed, infinite-dimensional subspace of its domain;
\item[(c)]\emph{strictly cosingular} (or \emph{Pe\l{}czy\'{n}ski}) if, for each infinite-codimensional, closed subspace $M$ of $F$, the operator $\pi T$ is not surjective, where $\pi\colon F\to F/M$ is the quotient operator.\end{itemize}

\begin{proposition}\label{slzw}${}$\begin{itemize}
\item[(i)]The operator ideal $\mathscr{W}$ is a proper subclass of $\mathscr{W\!C\!G}$.
\item[(ii)]The operator ideals $\mathscr{V}$, $\mathscr{S}$ and $\mathscr{C\!S}$ are incomparable to $\mathscr{W\!C\!G}$.\end{itemize}\end{proposition}
\begin{proof}(i) This is clear as explained above.

(ii) To see that $\mathscr{V}\not\subseteq \mathscr{W\!C\!G}$, recall that any operator on $\ell_1(\omega_1)$ is completely continuous, so that the identity on $\ell_1(\omega_1)$ belongs to $\mathscr{V}(\ell_1(\omega_1))\setminus \mathscr{W\!C\!G}(\ell_1(\omega_1))$. 

Conversely, to prove that $\mathscr{W\!C\!G}\not\subseteq \mathscr{V}$, by the Rosenthal--Dor $\ell_1$-theorem \cite[Chapter XI]{diestelseq} any completely continuous operator on a space without a copy of $\ell_1$ is compact. Thus, the identity operator on any infinite-dimensional $\mathsf{WCG}$ Banach space without a subspace isomorphic to $\ell_1$ belongs to $\mathscr{W\!C\!G}\setminus \mathscr{V}$.

The relations $\mathscr{W\!C\!G}\not\subseteq \mathscr{S},\,\mathscr{W\!C\!G}\not\subseteq \mathscr{C\!S}$ are clear, simply consider the identity operator on an infinite-dimensional $\mathsf{WCG}$ Banach space. 

Let $T\colon C[0,\omega_1]\to \ell_\infty([0,\omega_1])$ be the inclusion operator. Since $C[0,\omega_1]$ is not $\mathsf{WCG}$, the operator $T$ is not in $\mathscr{W\!C\!G}$. To prove that $\mathscr{W\!C\!G}\not\supseteq \mathscr{C\!S}$, it is enough to show that $T$ is strictly cosingular. In the light of a result by Bourgain and Diestel \cite{bourgain_diestel}, it suffices to ensure that $T^*$ maps weak*-null sequences into norm-null sequences. This is, however, automatic since weak*-null sequences in $\ell_\infty([0,\omega_1])^*$ converge weakly (this is the counterpart of \cite[Theorem VII.15]{diestelseq} for uncountable index sets), $T^*$ is weak-to-weak continuous like every operator and $\ell_1([0,\omega_1])\cong C[0,\omega_1]^*$ has the Schur property, which means that weakly convergent sequences in $C[0,\omega_1]^*$ converge in norm.

For the relation $\mathscr{W\!C\!G}\not\supseteq \mathscr{S}$, take a set $\Gamma$ such that there is a surjective operator $T\colon \ell_1(\Gamma)\to C[0,\omega_1]$. We have $T\notin \mathscr{W\!C\!G}$, because $C[0,\omega_1]$ is not weakly compactly generated. To see that $T$ is strictly singular, observe that if $T|_X$ was bounded below on some infinite-dimensional subspace $X$ of $\ell_1(\Gamma)$, one could find an isomorphic copy $Y$ of $c_0$ in $T(X)$ (as $C[0,\omega_1]$ is $c_0$-saturated, that is, each infinite-dimensional, closed subspace of $C[0,\omega_1]$ contains a subspace isomorphic to $c_0$; cf. also \cite{Lo}). Consequently, $(T|_X)^{-1}(Y)$ would be an isomorphic copy of $c_0$ in $\ell_1(\Gamma)$, which is impossible.\end{proof}

\section{Weakly compactly generated operators on the long James space}
In this section we prove that the ideal of weakly compactly generated operators is the unique maximal ideal of the algebra of bounded operators on the $p$th long James space $\J_p(\om_1)$. Then we derive from this fact several characterisations of this ideal. The long James space (originally for $p=2$) serves as a~counterexample to numerous questions in Banach space theory (consult Edgar's paper \cite{edgar} for details). 

Let $p\in (1,\infty)$. For any ordinal $\eta$ and any function $x\colon [0,\eta)\to\K$ define
$$\n{x}_{p,0}=\sup\Bigl\{\Bigl(\sum_{j=1}^n\abs{x(\alpha_j)-x(\alpha_{j-1})}^p\Bigr)^{1/p}\colon n\in\N\mbox{ and }0\leqslant\alpha_0<\alpha_1<\ldots <\alpha_n<\eta\Bigr\} .$$
Edgar \cite{edgar} defined the long James space to be
$$\J_p^{(0)}(\eta)=\bigl\{x\colon [0,\eta)\to\K\,\vert\,x\mbox{ is continuous}, x(0)=0\mbox{ and }\n{x}_{p,0}<\infty\bigr\}.$$
In fact, for our purposes we require a slight modification of Edgar's construction. Let $\eta$ be a non-zero limit ordinal. We set
 $$\w\J_p(\eta)=\bigl\{x\colon [0,\eta)\to\K\,\vert\,\lim_{\alpha\to\eta}x(\alpha)=0\mbox{ and }\n{x}_{p,0}<\infty\bigr\}$$
and define $\J_p(\eta)=\{x\in\w\J_p(\eta)\colon x\mbox{ is continuous}\}$. It turns out that all these three spaces are pairwise isomorphic. Indeed, the unique order preserving bijection $\p$, from $\eta$ onto the set $D(\eta)$ of all successors less than $\eta$, induces an isometry $U\colon\w\J_p(\eta)\to\J_p(\eta)$ via the formula
$$
U(x)(\alpha)=\left\{\begin{array}{ll}
x(\p^{-1}(\alpha)), & \mbox{if }\alpha\in D(\eta),\\
\lim_{\beta\to\alpha}x(\beta), & \mbox{if }\alpha\in\eta\setminus D(\eta),
\end{array}
\right.
$$
whereas the map $V\colon\J_p(\eta)\to\J_p^{(0)}(\eta)$, given by
$$
V(x)(\alpha)=\left\{\begin{array}{ll}
0, & \mbox{if }\alpha=0,\\
x(0)+x(\alpha), & \mbox{if }0<\alpha<\eta,\\
x(0), & \mbox{if }\alpha=\eta,
\end{array}
\right.
$$
yields an isomorphism between $\J_p(\eta)$ and $\J_p^{(0)}(\eta)$.

All these spaces may also be equipped with the norm 
$$\begin{array}{r}
\n{x}_{\J_p}=2^{-1/p}\sup\Bigl\{\Bigl(\abs{x(\alpha_n)-x(\alpha_0)}^p+\sum_{j=1}^n\abs{x(\alpha_j)-x(\alpha_{j-1})}^p\Bigr)^{1/p}\colon n\in\N\mbox{ and}\,\,\,\\
0\leqslant \alpha_0<\alpha_1<\ldots <\alpha_n<\eta\Bigr\},
\end{array}
$$
which is more natural than $\n{\cdot}_{p,0}$ in the sense that $\n{e_\alpha}_{\J_p}=1$ for every $\alpha<\eta$, where $e_\alpha=\ind_{\{\alpha\}}$. Moreover, for any $x\in\w\J_p(\eta)$ we have 
$$2^{-1/p}\n{x}_{p,0}\leqslant\n{x}_{\J_p}\leqslant 2^{1/p}\n{x}_{p,0}.$$

According to \cite[Propositions 1, 3]{edgar}, we know that
\begin{itemize*}
\item $(\ind_{(\alpha,\eta]})_{0\leqslant\alpha<\eta}$ is a basis for $\J_p^{(0)}(\eta)$;
\item $(e_\alpha^\ast)_{0<\alpha\leqslant\eta}$ is a basis for $\J_p^{(0)}(\eta)^\ast$,
\end{itemize*}
where $e_\alpha^\ast(f)=f(\alpha)\;(\alpha<\eta, f\in \J_p^{(0)}(\eta))$. By applying the isomorphism $V$ and the dual isomorphism $V^\ast$, we infer that
\begin{itemize*}
\item $(\ind_{[0,\alpha]})_{0\leqslant\alpha<\eta}$ is a basis for $\J_p(\eta)$;
\item $(e_\alpha^\ast)_{0\leqslant\alpha<\eta}$ is a basis for $\J_p(\eta)^\ast$.
\end{itemize*}

From now, we specialise to $\eta=\omega_1$, the smallest uncountable ordinal. Let us recall that $\J_p(\om_1)$ is isomorphic to its bidual and has the Radon--Nikod\'ym property, yet it is not isomorphic to a~subspace of a~$\mathsf{WCG}$ Banach space.

In the case of the classical James space $\J_p$ it was shown by Laustsen that $\W(\J_p)$ is the unique maximal ideal of $\B(\J_p)$ and, moreover, that an~operator on $\J_p$ is weakly compact if and only if it factors through the reflexive space $(\bigoplus_{n=1}^\infty\J_p(n))_{\ell_p}$, where $\J_p(n)=\span\{e_j\}_{j\leqslant n}$ (see \cite[Theorem 4.16]{laustsen} and \cite[Theorem 4.3]{laustsencommutator}, respectively). According to Willis (cf. \cite[Proposition 6]{williscompressible}), the ideal $\W(\J_p)$ may be also characterised as the ideal of compressible operators. Let us recall that for any Banach space $X$ an~operator $T\in\B(X)$ is said to be {\it compressible} if there is $n\in\N$, and a~sequence $(Q_k)_{k=1}^\infty$ of projections on $X^n$, such that $Q_kQ_\ell=0$ whenever $k\not=\ell$ and $T$ factors through $Q_k$ for each $k\in\N$. Equivalently (cf. \cite[Proposition 1]{williscompressible}), $T\in\B(X)$ is compressible if and only if there exist $n\in\N$ and sequences $(D_k)_{k=1}^\infty$ and $(E_k)_{k=1}^\infty$ of closed subspaces of $X^n$ such that
\begin{itemize*}
\item[(c1)] $X^n=D_1\oplus E_1$ and $E_k=D_{k+1}\oplus E_{k+1}$ for every $k\in\N$;
\item[(c2)] $T$ factors through $D_k$ for every $k\in\N$.
\end{itemize*}
We denote by $\mathscr{C}(X)$ the ideal of compressible operators on $X$.

The methods used by Laustsen and Willis to obtain characterisations of the ideal $\W(\J_p)$ are based on Lemma 2.1 from the paper \cite{loy_willis} by Loy and Willis. It asserts that every operator $T\in\W(\J_p)$ admits a~decomposition $T=K+R$, where $K$ is compact and $R$ has only finitely many non-zero entries in each line of its matrix representation, and also satisfies $R^{\ast\ast}(\ind_\N)=0$. This approach is rather useless for characterising the ideal $\WCG(\J_p(\om_1))$, since weakly compactly generated operators are not characterised by properties of their second adjoints, unlike weakly compact operators.

Our approach is based on Lemma 1.2 from \cite{kaniakoszmiderlaustsen} by the first-named author, Koszmider and Laustsen, which refines certain results of Alspach and Benyamini from \cite{AB}. We shall prove a~counterpart of that lemma for $\J_p(\om_1)$. Before proceeding to the proof let us note that, in view of the theorem of Hagler and Johnson \cite[Corollary 2]{hagler_johnson}, and the fact that $\J_p(\om_1)^\ast$ has the Radon--Nikod\'ym property, the unit ball of $\J_p(\om_1)^\ast$ is weak$^\ast$ sequentially compact. Though the general idea of the proof remains the same, some modification is needed, as the original argument heavily relies on the identification $C[0,\alpha]^\ast\cong\ell_1([0,\alpha])$ for any ordinal $\alpha$. 

\begin{theorem}\label{Koszmider_modified}
For every $p\in (1,\infty)$ and every $T\in\B(\J_p(\om_1))$ there exists a~$\lambda\in\K$ such that for some club set (a closed and unbounded set) $D\subseteq\om_1$ we have 
\begin{equation}\label{superlemma}e_\alpha^\ast T(x)=\lambda e_\alpha^\ast x\quad\mbox{for }x\in\J(\om_1)\mbox{ and }\alpha\in D.\end{equation}
\end{theorem}
\begin{proof}
For each $\alpha<\omega_1$ define $\p_\alpha=T^\ast e_\alpha^\ast$. We shall prove that for some $\lambda\in\K$ there is a~club subset $D\subset\omega_1$ with $\p_\alpha=\lambda e_\alpha^\ast$, since then $\langle Tx,e_\alpha^\ast \rangle =\langle x,T^\ast e_\alpha^\ast \rangle=\lambda e_\alpha^\ast x$ for $x\in X$, $\alpha\in D$, and the assertion would follow.

Let $\mathcal{Q}$ be any countable, dense subset of $\K$. For every $\alpha<\omega_1$ and $k\in\N$ we may find a~finite set $F_{\alpha,k}\subset\omega_1$ and scalars $q_{\alpha,k,\beta}\in\mathcal{Q}$ (for $\beta\in F_{\alpha,k}$) such that
\begin{equation}\label{ppp}
\n{\p_{\alpha,k}-\p_\alpha}<1/k\, ,\quad\mbox{where}\quad\p_{\alpha,k}=\sum_{\beta\in F_{\alpha,k}}q_{\alpha,k,\beta}e_\beta^\ast .
\end{equation}

\vspace*{2mm}\noindent
{\sc Claim 1. }If $(\alpha_k)_{k=1}^\infty\subset\omega_1$ and $\alpha_k\to\alpha\in\omega_1$, then $\p_{\alpha_k,k}\xrightarrow[]{w\ast}\p_\alpha$.

\vspace*{2mm}\noindent
Proof of Claim 1. For any $x\in\J_p(\om_1)$ we have $$\abs{(\p_{\alpha_k,k}-\p_\alpha)(x)}\leqslant\n{\p_{\alpha_k,k}-\p_{\alpha_k}}\cdot\n{x}_{\J_p}+\abs{(\p_{\alpha_k}-\p_\alpha)(x)}.$$The first term tends to zero by \eqref{ppp}. By the continuity of each $x\in\J_p(\om_1)$, we have $e_{\alpha_k}^\ast\xrightarrow[]{w\ast}e_\alpha^\ast$, thus the weak$^\ast$-continuity of $T^\ast$ gives $\p_{\alpha_k}=T^\ast e_{\alpha_k}^\ast\xrightarrow[]{w\ast}T^\ast e_\alpha^\ast=\p_\alpha$, which means that the second term tends to zero as well.

\vspace*{2mm}
For each $k\in\N$ the $\Delta$-system lemma produces an uncountable set $A_k\subset\omega_1$ and a~root $\Delta_k=\{\beta_{k,1},\ldots ,\beta_{k,\abs{\Delta_k}}\}\subset\omega_1$ such that
\begin{equation}\label{lambda_1}
F_{\alpha,k}\cap F_{\alpha^\prime,k}=\Delta_k\quad\mbox{for }\alpha, \alpha^\prime\in A_k, \alpha\not=\alpha^\prime .
\end{equation}
By deleting at most countably many elements from each of $A_k$'s we may also assume that:
\begin{itemize*}
\item[(a)] for every $k\in\N$ there is an $m_k\in\N$ with $\abs{F_{\alpha,k}}=m_k$ for each $\alpha\in A_k$;
\item[(b)] $\sup\ccup_{i\in\N}\Delta_i<\min(F_{\alpha,k}\setminus\Delta_k)$ for every $k\in\N$ and $\alpha\in A_k$;
\item[(c)] for every $k\in\N$ and $\alpha\in A_k$ there is an~order preserving bijection $\sigma_{\alpha,k}\colon [m_k]\to F_{\alpha,k}$ such that $\sigma_{\alpha,k}(i)=\beta_{k,i}$ for each $1\leqslant i\leqslant\abs{\Delta_k}$;
\item[(d)] for every $k\in\N$ there are scalars $q_{k,i}\in\mathcal{Q}$ (for $1\leqslant i\leqslant m_k$) such that $q_{\alpha,k,\sigma_{\alpha,k}(i)}=q_{k,i}$ for $\alpha\in A_k$ and $1\leqslant i\leqslant m_k$.
\end{itemize*}

\noindent
{\it Case 1. }$m_k=\abs{\Delta_k}$ for infinitely many $k$'s. 

Without loss of generality, we may suppose that $m_k=\abs{\Delta_k}$ for every $k\in\N$. Then for every $k\in\N$ and $\alpha\in A_k$ we have $F_{\alpha,k}=\Delta_k$ and $$\p_{\alpha,k}=\rho_k:=\sum_{1\leqslant i\leqslant\abs{\Delta_k}}q_{k,i}e_{\beta_{k,i}}^\ast .$$

We define the club $D\subset\omega_1$ as the set of all limits of sequences $(\alpha_k)_{k=1}^\infty$ with $\alpha_k\in A_k$ for $k\in\N$. Now, if $\alpha\in D$ is the limit of such a~sequence, then Claim 1 implies $\rho_k=\p_{\alpha_k,k}\overset{w\ast}{\longrightarrow}\p_\alpha$. Consequently, for all $\alpha\in D$ the functional $\p_\alpha$ is the same, and is equal to the weak$^\ast$ limit of $(\rho_k)_{k\in\N}$. Moreover, since each $x\in\J_p(\om_1)$ satisfies $\lim_{\alpha\to\om_1}x(\alpha)=0$, we have $\p_\alpha=T^\ast e_\alpha^\ast\xrightarrow[\alpha\to\om_1]{w\ast}T^\ast(0)=0$, thus $\p_\alpha=0$ for $\alpha\in D$ and our assertion is valid with $\lambda=0$.

\vspace*{2mm}\noindent
{\it Case 2. }$m_k>\abs{\Delta_k}$ for infinitely many $k$'s.

With no loss of generality we may suppose that $m_k>\abs{\Delta_k}$ for every $k\in\N$. For each $k\in\N$ and $\alpha\in A_k$ define $$\psi_{\alpha,k}=\sum_{\abs{\Delta_k}<j\leqslant m_k}q_{k,j}e_{\sigma_{\alpha,k}(j)}^\ast ,$$that is, $\p_{\alpha,k}=\rho_k+\psi_{\alpha,k}$. For every $k\in\N$ and $\alpha<\om_1$ we have $\n{\p_{\alpha,k}}\leqslant 1/k+\n{T^\ast e_\alpha^\ast}\leqslant 1+\n{T}$. Since the unit dual ball of $\J_p(\om_1)$ is weak$^\ast$ sequentially compact, we may find a~strictly increasing sequence $(k_i)_{i=1}^\infty$ of natural numbers, and a~strictly increasing sequence $(\beta_i)_{i=1}^\infty\subset\om_1$ such that $\beta_i\in A_{k_i}$ for each $i\in\N$ and $\p_{\beta_i,k_i}\xrightarrow[]{w\ast}\p_0$ for some $\p_0\in\J_p(\om_1)^\ast$.

\vspace*{2mm}\noindent
{\sc Claim 2. }There exist a~number $\lambda\in\K$ and a~functional $\rho\in\J_p(\om_1)^\ast$ such that $$\sum_{\abs{\Delta_{k_i}}<j\leqslant m_{k_i}}q_{k_i,j}\xrightarrow[i\to\infty]{}\lambda\quad\mbox{and}\quad\rho_{k_i}\xrightarrow[i\to\infty]{w\ast}\rho .$$

\vspace*{1mm}\noindent
Proof of Claim 2. Consider $x_0\in\J_p(\om_1)$ defined by $$x_0(\alpha)=\left\{\begin{array}{ll}1, & \mbox{if }\sup\bigcup_{i\in\N}\Delta_i<\alpha\leqslant\sup\bigcup_{i\in\N}F_{\beta_i,k_i},\\
0, & \mbox{otherwise.}\end{array}\right.$$Plainly, $\rho_{k_i}(x_0)=0$ and $\psi_{\beta_i,k_i}(x_0)=\sum_{\abs{\Delta_{k_i}}<j\leqslant m_{k_i}}q_{k_i,j}$ for every $i\in\N$, thus the convergence $\p_{\beta_i,k_i}(x_0)\to\p_0(x_0)$ implies the first part of the claim with $\lambda=\p_0(x_0)$. To complete the argument let $$X=\oo\span\Bigl\{\ind_{[0,\alpha]}\colon 0\leqslant\alpha\leqslant\sup\ccup_{i\in\N}\Delta_i\Bigr\}$$and observe that for each $x\in X$ we have $\rho_{k_i}(x)=\p_{\beta_i,k_i}(x)\to\p_0(x)$. Obviously, $\J_p(\om_1)=X\oplus Y$, where $Y$ consists of all sequences $x\in\J_p(\om_1)$ with $x(\alpha)=0$ for each $\alpha\leqslant\sup\ccup_{i\in\N}\Delta_i$. For $y\in Y$ and every $i\in\N$ we have $\rho_{k_i}(y)=0$, hence our assertion holds true with $\rho$ defined by $\rho(x+y)=\p_0(x)$ for $x\in X$ and $y\in Y$.

\vspace*{1mm}
Now, define $D\subset\om_1$ to be the set of all ordinals $\alpha\in\om_1$ for which there exists a~sequence $(\alpha_i)_{i=1}^\infty\subset\om_1$ satisfying:
\begin{itemize*}
\item[(1)] $(\alpha_i)_{i=1}^\infty$ is strictly increasing;
\item[(2)] $\alpha_i\in A_{k_i}$ for each $i\in\N$;
\item[(3)] $\lim_{i\to\infty}\alpha_i=\lim_{i\to\infty}\min(F_{\alpha_i,k_i}\setminus\Delta_{k_i})=\alpha$;
\item[(4)] $\max(F_{\alpha_i,k_i})<\alpha_{i+1}<\alpha$ for each $i\in\N$.

\end{itemize*}
It is clear that $D$ is then a~club subset of $\om_1$.

\vspace*{2mm}\noindent
{\sc Claim 3. }If $(\alpha_i)_{i=1}^\infty\subset\om_1$ satisfies conditions (1)-(4) and converges to an~$\alpha\in D$, then $\p_{\alpha_i,k_i}\xrightarrow[i\to\infty]{w\ast}\rho+\lambda e_\alpha^\ast$.

\vspace*{1mm}\noindent
Proof of Claim 3. For every $x\in\J_p(\om_1)$ we have
\begin{equation*}
\begin{split}
\abs{\p_{\alpha_i,k_i}(x)-(\rho+\lambda e_\alpha^\ast)(x)}&\leqslant\abs{(\rho_{k_i}-\rho)(x)}+\abs{\psi_{\alpha_i,k_i}(x)-\lambda e_\alpha^\ast(x)}\\
&\leqslant\abs{(\rho_{k_i}-\rho)(x)}+\Bigl|\sum_{\abs{\Delta_{k_i}}<j\leqslant m_{k_i}}q_{k_i,j}e_\alpha^\ast(x)-\lambda e_\alpha^\ast(x)\Bigr|\\
&\qquad\qquad\, +\Bigl|\sum_{\abs{\Delta_{k_i}}<j\leqslant m_{k_i}}q_{k_i,j}\bigl(e_{\sigma_{\alpha_i,k_i}(j)}^\ast(x)-e_\alpha^\ast(x)\bigr)\Bigr|.
\end{split}
\end{equation*}
By Claim 2, the first two terms tend to zero as $i\to\infty$.

Let $X_\alpha$ be the set of all sequences $x\in\J_p(\om_1)$ which are constant on some neighbourhood of $\alpha$. If $x\in X_\alpha$ then conditions (3) and (4) guarantee that for sufficiently large $i$'s every summand in the last term equals zero. Consequently, $\p_{\alpha_i,k_i}(x)\to(\rho+\lambda e_\alpha^\ast)(x)$ is valid for every $x\in X_\alpha$, hence also for every $x\in\span X_\alpha$, so also for $x\in\oo\span X_\alpha$, since $\p_{\alpha_i,k_i}$ are equicontinuous. But $X_\alpha$ is linearly dense in $\J_p(\om_1)$, as it contains each element of the Schauder basis.

\vspace*{2mm}
To complete the proof let again $\alpha\in D$ be the limit of a~sequence $(\alpha_i)_{i=1}^\infty\subset\om_1$ satisfying (1)-(4). By Claim 1, we have $\p_{\alpha_i,k_i}\xrightarrow[]{w\ast}\p_\alpha$, whence Claim 3 yields $\p_\alpha=\rho+\lambda e_\alpha^\ast$. 

Since this is true for every $\alpha\in D$, we may pass to the limit with $\alpha\to\om_1$ ($\alpha\in D$) to get $\p_\alpha\xrightarrow[]{w\ast}\rho$ and, on the other hand, $\p_\alpha=T^\ast e_\alpha^\ast\xrightarrow[]{w\ast}0$. Therefore, $\rho=0$, thus $\p_\alpha=\lambda e_\alpha^\ast$ for each $\alpha\in D$.
\end{proof}

Now, as in \cite{kaniakoszmiderlaustsen}, we define a~map $\Lambda_p\colon\B(\J_p(\om_1))\to\K$ by $\Lambda_p(T)=\lambda$, where $\lambda\in\K$ is chosen such that (\ref{superlemma}) holds (the uniqueness of such a~$\lambda$, for fixed $T$, follows from the fact that the intersection of two club subsets of $\om_1$ is again a~club subset). Obviously, $\Lambda_p$ is a~non-zero linear and multiplicative functional, hence $\ker\Lambda_p$ is~a (maximal) ideal in $\B(\J_p(\om_1))$ of codimension one.

Let $L(\om_1)$ be the set of all non-zero limit ordinals less than $\om_1$. For every ordinal $\alpha\in (0,\om_1)$ we define a~subspace of $\J_p(\om_1)$ by $$\J_p(\alpha)=\oo\span\bigl\{\ind_{[0,\beta]}\colon 0\leqslant\beta<\alpha\bigr\}$$and we let $$\mathfrak{G}_p=\Bigl(\bigoplus_{\alpha\in L(\om_1)}\J_p(\alpha)\Bigr)_{\ell_p}.$$Being an~$\ell_p$-sum of $\mathsf{WCG}$ (even separable) Banach spaces, with $p\in (1,\infty)$, the space $\mathfrak{G}_p$ is a~$\mathsf{WCG}$ Banach space (cf. \cite[Proposition 2.4]{LI}). It turns out that weakly compactly generated operators on the long James space factor through this concrete $\mathsf{WCG}$ space, just like weakly compact operators on the classical James space factor through the concrete reflexive space $\left(\bigoplus_{n=1}^\infty \J_p(n)\right)_{\ell_p}$ identified by Laustsen.

The estimates given by Casazza, Lin and Lohman in \cite[Lemma 2]{casazza_lin_lohman} for the classical James space and $p=2$ can be easily generalised in the following manner.
\begin{lemma}\label{Casazza}
Let $p\in (1,\infty)$ and $k,n_1,\ldots ,n_k\in\N$. For any ordinal numbers $$\gamma_{1,1}<\ldots<\gamma_{1,n_1}<\ldots<\gamma_{k,1}<\ldots<\gamma_{k,n_k}<\om_1,$$satisfying $\gamma_{i,n_i}+1<\gamma_{i+1,1}$ for each $1\leqslant i<k$, and for any scalars $t_{i,j}$, we have 
$$
\sum_{i=1}^k\Bigl\|\sum_{j=1}^{n_i}t_{i,j}e_{\gamma_{i,j}}\Bigr\|_{p,0}^p\leqslant\Bigl\|\sum_{i=1}^k\sum_{j=1}^{n_i}t_{i,j}e_{\gamma_{i,j}}\Bigr\|_{p,0}^p\leqslant 2^{p-1}\sum_{i=1}^k\Bigl\|\sum_{j=1}^{n_i}t_{i,j}e_{\gamma_{i,j}}\Bigr\|_{p,0}^p.
$$
\end{lemma}
\begin{proof}
The first inequality is obvious, since for each $1\leqslant i\leqslant k$ the $i$th summand on the left-hand side may be calculated using only indices from the interval $(\gamma_{i-1,n_{i-1}},\gamma_{i+1,1})$, where we put $\gamma_{0,n_0}=-1$ and $\gamma_{k+1,1}=\om_1$.

For the second estimate notice that convexity of the function $x\mapsto |x|^p$ gives $\abs{t+u}^p\leqslant 2^{p-1}(\abs{t}^p+\abs{u}^p)$ for $t,u\in\C$. Consequently, for $t,u\in\mathbb{C}$ we have $$\abs{t-u}^p\leqslant 2^{p-1}\bigl(\abs{t-0}^p+\abs{0-u}^p\bigr),$$thus we can change every estimate under the supremum sign defining the middle term into a~sum which does not exceed the right-hand side.
\end{proof}

\begin{proposition}\label{prop_long_James}
For each $p\in (1,\infty)$ the space $\J_p(\om_1)$ contains a~complemented copy of $\mathfrak{G}_p$ with $\J_p(\om_1)\simeq\mathfrak{G}_p\oplus\J_p(\om_1)$.
\end{proposition}
\begin{proof}
For $\alpha<\om_1$ let us define a~subspace of $\w\J_p$ by $\w\J_p(\alpha)=\oo\span\{e_\beta\colon 0\leqslant\beta<\alpha\}$. The order preserving bijection $\p\colon\om_1\to D(\om_1)\cup\{0\}$ induces a surjective isometry between $\w\J_p(\alpha)$ and $\J_p(\alpha)$ for each $\alpha\in L(\om_1)$. Hence, $\mathfrak{G}_p$ is isometric to $\mathfrak{\w G}_p$, where $$\mathfrak{\w G}_p=\Bigl(\bigoplus_{\alpha\in L(\om_1)}\w\J_p(\alpha)\Bigr)_{\ell_p}.$$It is enough to prove that the space $\w\J_p(\om_1)$ contains a~complemented copy of $\mathfrak{\w G}_p$. 

\vspace*{2mm}\noindent
{\sc Claim 1. }The canonical basis of $\mathfrak{\w G}_p$ is equivalent to a~basic sequence $(e_\alpha)_{\alpha\in A}\subset\w\J_p(\om_1)$ with a~certain set $A\subset\om_1$.

\vspace*{2mm}\noindent
Proof of Claim 1. Define $$\Gamma=\{(\alpha,\beta)\in\om_1\times\om_1\colon\alpha\in L(\om_1)\mbox{ and }0\leqslant\beta<\alpha\}$$and let $\Gamma^\ast=\Gamma\cup\om_1$. Consider a~linear order $\prec$ on $\Gamma^\ast$ defined as follows: 
\begin{itemize*}
\item[(i)] $\prec\vert_\Gamma$ is the lexicographic order;
\item[(ii)] $\prec\vert_{\om_1}$ is the natural order;
\item[(iii)] $(\beta,\gamma)<\alpha<(\delta,\e)$ for every $(\beta,\gamma), (\delta,\e)\in\Gamma$ with $\beta\leqslant\alpha<\delta$.
\end{itemize*}
By a~standard recursive argument, we infer that $(\Gamma^\ast,\prec)$ is order-isomorphic to $\om_1$. Let $\psi\colon\Gamma^\ast\to\om_1$ be the order-isomorphism.

Now, suppose $k,n_1,\ldots ,n_k\in\N$ and we are given ordinal numbers $\alpha_1<\ldots<\alpha_k\in L(\om_1)$ and $\beta_{i,1}<\ldots<\beta_{i,n_i}<\alpha_i$ (for $1\leqslant i\leqslant k$) such that $(\alpha_i,\beta_{i,j})\in\Gamma$. Let $x$ be the element of $\mathfrak{\w G}_p$ whose $\alpha_i$th coordinate equals $\sum_{j=1}^{n_i}t_{i,j}e_{\beta_{i,j}}$ for $1\leqslant i\leqslant k$ (with some scalars $t_{i,j}$), and whose all other coordinates are zeros. Let also $\gamma_{i,j}=\psi(\alpha_i,\beta_{i,j})$ for $1\leqslant i\leqslant k$ and $1\leqslant j\leqslant n_i$.

Define $y\in\w\J_p(\om_1)$ by the formula $y=\sum_{i=1}^k\sum_{j=1}^{n_i}t_{i,j}e_{\gamma_{i,j}}$. Since $$\gamma_{i,n_i}=\psi(\alpha_i,\beta_{i,n_i})<\psi(\alpha_i)<\psi(\alpha_{i+1},\beta_1)=\gamma_{i+1,1}\quad\mbox{for }1\leqslant i<k,$$an application of Lemma \ref{Casazza} yields $$\n{x}_{\mathfrak{\w G}_p}^p\leqslant\n{y}_{p,0}^p\leqslant 2^{p-1}\n{x}_{\mathfrak{\w G}_p}^p.$$
Thus there is an isomorphism witnessing that the canonical basis of $\mathfrak{\w G}_p$ is equivalent to $(e_\alpha)_{\alpha\in A}$ with $A=\psi(\Gamma)\subset\om_1$.

The next claim will complete the proof.

\vspace*{2mm}\noindent
{\sc Claim 2. }The subspace $X=\oo\span\{e_\alpha\colon\alpha\in A\}$ is complemented in $\w\J_p(\om_1)$ by a~copy of $\w\J_p(\om_1)$.

\vspace*{2mm}\noindent
Let $B=(\om_1\setminus A)\cup\{-1\}$. Following the lines of the proof of \cite[Theorem 5]{casazza_lin_lohman}, consider two sets: 
$$\begin{array}{r}
C=\bigl\{\ind_{[\alpha,\beta]}\in\w\J_p(\om_1)\colon\alpha\in A, \mbox{ either }\alpha\in L(\om_1)\mbox{ or }\alpha=\alpha^\prime+1\mbox{ with }\alpha^\prime\in B,\,\,\,\\
\mbox{whereas }\beta=\min\{\beta^\prime\in B\colon \beta^\prime>\alpha\}\bigr\}
\end{array}
$$
and $$D=\{e_\beta\in\w\J_p(\om_1)\colon\beta\in B, \mbox{either }\beta\in L(\om_1)\mbox{ or }\beta=\beta^\prime+1\mbox{ with }\beta^\prime\in B\},$$and define $Y=\oo\span(C\cup D)$. Then for $x\in X$ and $y\in Y$ we have $\n{y}_{p,0}\leqslant\n{x+y}_{p,0}$, since every partial variation approximating $\n{y}_{p,0}$ may be calculated for coordinates from $B$, and it remains the same for $x+y$. Therefore, $\n{x}_{p,0}\leqslant 2\n{x+y}_{p,0}$, thus there is a~projection $P$ on $\w\J_p(\om_1)$, with range $X$ and kernel $Y$, and $\n{P}\leqslant 2$.

The elements of $C\cup D$ form a~block basic sequence of $(e_\alpha)_{\alpha<\om_1}$ which is plainly equivalent to $(e_\alpha)_{\alpha<\om_1}$. Thus $Y\simeq\w\J_p(\om_1)$ and the proof is completed.
\end{proof}

\begin{lemma}\label{l_p-suma}
For each $p\in (1,\infty)$ we have $\mathfrak{G}_p\simeq\ell_p(\om_1,\mathfrak{G}_p)$ (the $\ell_p$-sum of $\om_1$ copies of $\mathfrak{G}_p$). Consequently, by Pe\l czy\'nski decomposition method, $\mathfrak{G}_p$ is also isomorphic to the $\ell_p$-sum of countably many copies of itself.
\end{lemma}
\begin{proof}
By using transfinite induction, with respect to the lexicographic order on the set $\om_1\times L(\om_1)$, we construct a~one-to-one map $\theta\colon\om_1\times L(\om_1)\to L(\om_1)$ such that $\theta(\alpha,\beta)>\beta$ for every $(\alpha,\beta)\in\om_1\times L(\om_1)$.

For $(\alpha,\beta)\in\om_1\times L(\om_1)$ let $\iota_{\alpha,\beta}\colon\J_p(\beta)\to\J_p(\theta(\alpha,\beta))$ be the embedding of the $\beta$th summand of the $\alpha$th coordinate of $\ell_p(\om_1,\mathfrak{G}_p)$ into $\J_p(\theta(\alpha,\beta))$ which just puts the sequences from $\J_p(\beta)$ into the first $\beta$ coordinates. Let $\iota\colon\ell_p(\om_1,\mathfrak{G}_p)\to\mathfrak{G}_p$ be the embedding naturally produced by all $\iota_{\alpha,\beta}$'s.

Obviously, $\iota$ is an isometric embedding, thus $\iota(\ell_p(\om_1,\mathfrak{G}_p))$ is a~closed subspace of $\mathfrak{G}_p$. Since every summand $\J_p(\theta(\alpha,\beta))$ of $\mathfrak{G}_p$ admits a~natural, norm one projection onto $\J_p(\beta)$, there also exists a~norm one projection from $\mathfrak{G}_p$ onto $\iota(\ell_p(\om_1,\mathfrak{G}_p))$. Thus, $\ell_p(\om_1,\mathfrak{G}_p)$ is isomorphic to a~complemented subspace of $\mathfrak{G}_p$ and, obviously, vice versa. Moreover, both spaces $\mathfrak{G}_p$ and $\ell_p(\om_1,\mathfrak{G}_p)$ are clearly isomorphic to their squares. Hence, by the Pe\l czy\'nski decomposition method, we get the assertion.
\end{proof}

The following two assertions can be proved in the same manner as Proposition 2.1 and Corollary 2.3 in \cite{kaniakoszmiderlaustsen}, so we omit their proofs.
\begin{lemma}\label{r(T)}
If $T\in\B(\J_p(\om_1))$ and $\Lambda_p(T)\not=0$ then $T$ fixes a~complemented copy of $\J_p(\om_1)$ and the range of $T$ contains a~copy of $\J_p(\om_1)$, complemented in $\J_p(\om_1)$.
\end{lemma}

\begin{corollary}\label{komplementarny_James}
If $Y\subset\J_p(\om_1)$ and $Y\simeq\J_p(\om_1)$ then there is a~subspace $Z\subseteq Y$ which is complemented in $\J_p(\om_1)$ and such that $Z\simeq\J_p(\om_1)$.
\end{corollary}

Before we proceed to the main result of this section, we require another piece of notation.

Following Dosev and Johnson \cite{dj}, for each Banach space $X$ we define
\[\mathscr{M}_X=\big\{T\in \mathscr{B}(X)\colon I \neq ATB\;\;\big(A,B\in \mathscr{B}(X)\big)\big\}.\]
In general, the set $\mathscr{M}_X$ need not be closed under addition but when it is, it is also the unique maximal ideal of $\mathscr{B}(X)$. A recent result of the first-named author and Laustsen \cite[Theorem 1.2]{kanialaustsen} states that $\mathscr{M}_{C[0,\omega_1]}$ is the unique maximal ideal of $\mathscr{B}(C[0,\omega_1])$.

Now, we are ready to prove the following theorem.

\begin{theorem}\label{long_James}
For any $p\in (1,\infty)$ we have $$\ker\Lambda_p=\WCG(\J_p(\om_1))=\mathscr{G}_{\mathfrak{G}_p}(\J_p(\om_1))=\mathscr{C}(\J_p(\om_1))=\mathscr{M}_{\J_p(\om_1)},$$
and this is the unique maximal ideal of $\B(\J_p(\om_1))$. Moreover, an~operator $T\in\B(\J_p(\om_1))$ belongs to this ideal if and only if it is $\J_p(\om_1)$-singular, that is, $T$ does not fix a~copy of $\J_p(\om_1)$.
\end{theorem}
\begin{proof}
Let us start by showing that the ideal $\ker\Lambda_p$ is contained in each of the remaining sets in the desired equality. So, suppose $T\in\B(\J_p(\om_1))$ and $\langle e_\alpha^\ast,  Tx\rangle =0$ for every $x\in\J_p(\om_1)$ and $\alpha\in D$, where $D$ is a~certain club subset of $\om_1$. We may assume that $D\subseteq L(\om_1)$.

Define $Y$ to be the set of all sequences $x\in\J_p(\om_1)$ such that $x(\alpha)=0$ for every $\alpha\in D$. Obviously, the range of $T$ lies in $Y$. By Lemma \ref{Casazza}, the space $Y$ is in turn isomorphic to $(\bigoplus_{\alpha<\om_1}\w\J_p(\mbox{ord}\,O_\alpha))_{\ell_p}$, where $\{O_\alpha\colon \alpha<\omega_1\}$ is the family of consecutive order-components of $\omega_1\setminus D$ (this family is uncountable as $D\subset L(\omega_1)$) and $\mbox{ord}\,O_\alpha$ denotes the order type of $O_\alpha$ ($\alpha<\omega_1$). That space is isomorphic to a~complemented subspace of $\mathfrak{\w G}_p$ given by \eqref{gtilde}. Hence, by Proposition \ref{prop_long_James}, it is isomorphic to a~subspace of $\J_p(\om_1)$ and is, of course, weakly compactly generated. Thus, the inclusion $\ker\Lambda_p\subset\WCG(\J_p(\om_1))$ has been proved. Furthermore, $\Lambda_p(T)=0$ implies $T\in\mathscr{G}_{\mathfrak{G}_p}(\J_p(\om_1))$, as in this case the range of $T$ is contained in a subspace of $\mathfrak{G}_p$.

Now, observe that Proposition \ref{prop_long_James} yields $$\J_p(\om_1)\cong\mathfrak{G}_p\oplus\J_p(\om_1)\cong\mathfrak{G}_p\oplus\mathfrak{G}_p\oplus\J_p(\om_1)\cong\ldots ,$$thus every operator factoring through $\mathfrak{G}_p$ is compressible. Consequently, $\ker\Lambda_p\subset\mathscr{C}(\J_p(\om_1))$.

To show the last of the announced inclusions recall that $\ker\Lambda_p\subset\mathscr{W\!C\!G}(J_p(\om_1))$ and observe that the identity of $\J_p(\om_1)$ cannot factor through a~weakly compactly generated operator, as $J_p(\om_1)$ is not weakly compactly generated. Hence, $\ker\Lambda_p\subseteq\mathscr{M}_{\J_p(\om_1)}$.

Since $\ker\Lambda_p$ is a~maximal ideal, and the ideals $\WCG(\J_p(\om_1))$ and $\mathscr{G}_{\mathfrak{G}_p}(\J_p(\om_1))$ are clearly proper, we get the first two of the claimed equalities. 

To show the equality $\ker\Lambda_p=\mathscr{C}(\J_p(\om_1))$ we shall prove that the identity operator $I_{\J_p(\om_1)}$ is not compressible.

By \cite[Proposition 2]{williscompressible}, that would be the case if and only if for a~certain $n\in\N$ we had a~decomposition
\begin{equation}\label{dec}
\J_p(\om_1)^n\cong Z\oplus\J_p(\om_1)^{n+1}
\end{equation}
with some Banach space $Z$. However, as the ideal $\WCG(\J_p(\om_1))$ has codimension one in $\B(\J_p(\om_1))$ hence there is no such a decomposition. 

For the equality $\ker\Lambda_p=\mathscr{M}_{\J_p(\om_1)}$ it remains to prove the inclusion \lq\lq$\supset$\rq\rq, but this will follow from the fact that any operator $T$ with $\Lambda_p(T)\not=0$ must fix a~copy of $\J_p(\om_1)$ and, consequently, the identity factors through $T$. These two statements are contained in the remaining part of the proof.

Now, notice that since $\Lambda_p(T)=0$ implies $T\in\WCG(\J_p(\om_1))$, it implies also that $T$ does not fix a~copy of $\J_p(\om_1)$. The converse follows from Lemma \ref{r(T)}. Hence, each of the ideals listed in our assertion is just the set of operators not fixing a~copy of $\J_p(\om_1)$. We conclude here that $\ker\Lambda_p=\mathscr{M}_{\J_p(\om_1)}$ is the unique maximal ideal of $\mathscr{B}(J_p(\omega_1))$.\end{proof}

\begin{corollary}\label{rangeofwcg}
The range of a weakly compactly generated operator on $\J_p(\om_1)$ is contained in a complemented, $\mathsf{WCG}$ subspace of $\J_p(\om_1)$.
\end{corollary}

Let us say that a sequence $(A_\xi)_{\xi <\omega_1}$ of subsets of $\omega_1$ is \emph{skipped} if for each pair $\xi_1< \xi_2<\omega_1$ we have $\sup A_{\xi_1}+1<\min A_{\xi_2}$. A sequence $(f_\xi)_{\xi<\omega_1}$ of functions defined on $\omega_1$ is \emph{skipped} provided the sequence $(\mbox{supp}\,f_\xi)_{\xi<\omega_1}$ is skipped. 

The next lemma is an easy consequence of Lemma \ref{Casazza}.
\begin{lemma}\label{skipped}For each $p\in (1,\infty)$ any normalised skipped sequence  $(f_\xi)_{\xi<\omega_1}$ in $\mathcal{J}_p(\omega_1)$ is equivalent to the canonical basis of $\ell_p(\omega_1)$.\end{lemma}

For each $\sigma<\omega_1$ let us define $P_\sigma f = f\cdot \mathbf{1}_{[0,\sigma]},\, f\in \mathcal{J}_p(\omega_1)$. We note that $P_\sigma$ is a well-defined contractive projection on $\mathcal{J}_p(\omega_1)$. The proof of the next theorem is based on ideas from the proofs of \cite[Theorem 1.3, Lemma 4.3]{kanialaustsen}.
\begin{theorem}\label{thmseprange}
  Let $p\in (1,\infty)$. The following assertions are equivalent for an operator $T$ on
  $\mathcal{J}_p(\omega_1)\colon$
  \begin{alphenumerate}
  \item\label{thmseprange2e} $T = TP_\sigma$ for some countable ordinal~$\sigma;$
  \item\label{thmseprange2g}
    $T\in\mathscr{G}_{\mathcal{J}_p(\sigma)}(\J_p(\om_1))$ for some
    countable ordinal~$\sigma;$
  \item\label{thmseprange2f}
    $T\in\overline{\mathscr{G}}_{\mathcal{J}_p(\sigma)}(\J_p(\om_1))$ for
    some countable ordinal~$\sigma;$
  \item\label{thmseprange2a} $T\in\mathscr{X}(\J_p(\om_1));$
  \item\label{thmseprange2d} $T$ does not fix a copy of~$\ell_p(\omega_1)$.
  \end{alphenumerate}
\end{theorem}
\begin{proof}The implications (\ref{thmseprange2e}) $\Rightarrow$ (\ref{thmseprange2g}) $\Rightarrow$ (\ref{thmseprange2f}) $\Rightarrow$ (\ref{thmseprange2a}) $\Rightarrow$ (\ref{thmseprange2d}) are clear. 

Assume contrapositively that (\ref{thmseprange2e}) fails. We claim that there is $\delta>0$ such that for each $\xi<\omega_1$, there is $f_\xi\in \mathcal{J}_p(\omega_1)$ with $\mbox{supp}f_\xi \subseteq (\xi, \omega_1),\, \|Tf\|\geqslant \delta$ and $\|f\|\leqslant 1$. Suppose this is not the case. Then, for $\delta_n=1/n$ we obtain a sequence $(\xi_n)_{n\in \mathbb{N}}$ of countable ordinals such that $\|Tf\|<1/n$ for each $f\in \mathcal{J}_p(\omega_1)$ with $\mbox{supp}f\subseteq (\xi_n, \omega_1)$. Let $\xi = \sup\{\xi_n\colon n\in \mathbb{N}\}$. Certainly, $\xi<\omega_1$. Take $g\in\mathcal{J}_p(\omega_1)$ with $\|(I - P_\xi)g\|\leqslant 1$. Letting $f=(I-P_\xi)g$ we infer that $\mbox{supp}f \subseteq (\xi, \omega_1)=\bigcap_{n\in\mathbb{N}}(\xi_n, \omega_1)$ as $P_\xi(I-P_\xi)=0$. Thus, $\|Tf\|<1/n$ for each $n\in \mathbb{N}$, so $0=Tf=T(I-P_\xi)g$, which proves $T=TP_\xi$, against the assumption.

Similarly, as in the proof \cite[Theorem 1.3]{kanialaustsen}, we choose inductively a normalised skipped sequence $(f_\xi)_{\xi<\omega_1}$ with $\|Tf_\xi\|\geqslant \delta$, where $\delta$ is as above. By Lemma \ref{skipped}, the subspace $X=\overline{\mbox{span}}\{f_\xi\colon \xi<\omega_1\}$ is isomorphic to $\ell_p(\omega_1)$. We note that $T|_X$ is bounded below, hence the proof of the implication (\ref{thmseprange2d}) $\Rightarrow$ (\ref{thmseprange2e}) is complete. \end{proof}

An element $x$ of an algebra $\mathscr{A}$ is a \emph{commutator} if there exist some $a,b\in \mathscr{A}$ such that $x = ab - ba$. It is well-known that if $\mathscr{A}$ is a unital Banach algebra, then its unit cannot be a commutator. The following result is the counterpart for $\J_p(\om_1)$ of \cite[Theorem 4.6]{laustsencommutator} and \cite[Theorem 5.1]{kaniakoszmiderlaustsen} for $C[0,\omega_1]$.

\begin{theorem}\label{commutator}Each operator in $\mathscr{W\!C\!G}(\mathcal{J}_p(\omega_1))$ is a sum of three commutators.\end{theorem}
\begin{proof}We have that $\mathscr{W\!C\!G}(\mathcal{J}_p(\omega_1))=\mathscr{G}_{\mathfrak{G}_p}(\J_p(\omega_1))$ and, by Lemma \ref{l_p-suma}, the space $\mathfrak{G}_p$ is isomorphic to its $\ell_p$-sum, so \cite[Proposition 3.7]{laustsencommutator} yields that each operator on $\mathfrak{G}_p$ is a sum of two commutators. Consequently, it follows from \cite[Lemma 4.5]{laustsencommutator} that each operator which factors through $\mathfrak{G}_p$ is a sum of three commutators.\end{proof}

Let $\vartheta\colon \mathscr{A}\to \mathscr{C}$ be a homomorphism between Banach algebras. We call
\[\mathcal{I}:=\{a\in \mathscr{A}\colon \mbox{the maps }b\mapsto \vartheta(ab),\,b\mapsto \vartheta(ba)\mbox{ are continuous}\}\]
\emph{the continuity ideal} of $\vartheta$. Certainly, $\mathcal{I}$ is a two-sided ideal of $\mathscr{A}$. Suppose $\mathscr{A}$ is an ideal of $\mathscr{B}(E)$, where $E$ is some Banach space. Willis \cite[Proposition 7]{williscompressible} proved that $\mathscr{A}\cdot\mathscr{C}(E)\cdot\mathscr{A}\subseteq \mathcal{I}$ (recall that $\mathscr{C}(E)$ stands for the ideal of compressible operators). Furthermore, he used this fact to prove that each homomorphism from $\mathscr{B}(\mathcal{J}_2)$ is continuous (in fact, his argument extends to arbitrary $p\in (1,\infty)$). A key-ingredient in Willis' proof is existence of a bounded right approximate identity in the ideal of weakly compact operators on the James space; we shall prove that every homomorphism from $\mathscr{B}(\mathcal{J}_p(\omega_1))$ is continuous without appealing to those type of results.
\begin{theorem}Every homomorphism from $\mathscr{B}(\mathcal{J}_p(\omega_1))$ is automatically continuous.\end{theorem}
\begin{proof}Let $\mathscr{C}$ be a Banach algebra and let $\vartheta\colon \mathscr{B}(\mathcal{J}_p(\omega_1))\to \mathscr{C}$ be a homomorphism. As $\mathscr{W\!C\!G}(\mathcal{J}_p(\omega_1))$ is equal to the ideal of compressible operators on $\mathcal{J}_p(\omega_1)$, the aforementioned result of Willis yields that $\mathscr{W\!C\!G}(\mathcal{J}_p(\omega_1))$ is contained in the continuity ideal $\mathcal{I}$ of $\vartheta$.

Since the ideal of weakly compactly generated operators on $\mathcal{J}_p(\omega_1)$ is of codimension one in $\mathscr{B}(\J_p(\om_1))$, it is sufficient to prove that $\vartheta$ restricted to $\mathscr{W\!C\!G}(\mathcal{J}_p(\omega_1))$ is continuous.

Let $(T_n)_{n=1}$ be a sequence of $\mathsf{WCG}$ operators on $\J_p(\om_1)$. We shall exhibit a weakly compactly generated operator $P$ such that $T_n = PT_n$ ($n\in \mathbb{N}$), which by definition of the continuity ideal, would complete the proof.

For each $n\in \mathbb{N}$ there is a weakly compactly generated subspace $X_n$ of $\mathcal{J}_p(\omega_1)$, containing the image of $T_n$ and which is, by Corollary \ref{rangeofwcg}, isomorphic to $\mathfrak{G}_p$. Let $X = \left( \bigoplus_{n\in \mathbb{N}}X_n \right)_{\ell_p}$. By Lemma \ref{l_p-suma}, $X$ isomorphic to $\mathfrak{G}_p$, so it is also $\mathsf{WCG}$. Define $J\colon X\to \mathcal{J}_p(\omega_1)$ by $J(x_n)_{n=1}^\infty = \sum_{n=1}^\infty x_n/n^p$. Let $U\colon Y\to X$ be an isomorphism, where $Y$ is a complemented subspace of $\mathcal{J}_p(\omega_1)$, isomorphic to $\mathfrak{G}_p$. Let $V=JU$, $V\colon Y\to \mathcal{J}_p(\omega_1)$. Because $Y$ is complemented in $\J_p(\om_1)$, we may extend $V$ to $\mathcal{J}_p(\omega_1)$; let $V$ still stand for any fixed such an extension. Plainly, $V\in \mathscr{W\!C\!G}(\mathcal{J}_p(\omega_1))$, so $\Lambda_p(V)=0$. Thus, there is a complemented copy $Z$ of $\mathfrak{G}_p$ containing the range of $V$. Let $P$ be a projection onto $Z$. It remains to notice that 
\[\bigcup_{n\in \mathbb{N}}\mbox{im}\,T_n\subseteq \bigcup_{n\in \mathbb{N}}X_n\subseteq \mbox{im}\,T=\mbox{im}\,V\subseteq Z.\]
In particular, $T_k=PT_k$ for each $k\in \mathbb{N}$.\end{proof}

Now, consider the complex version of $\mathcal{J}_p(\omega_1)$ only. Recall that a linear (not necessarily bounded) functional $\tau$ on a complex algebra $\mathscr{A}$ is a trace if $\tau(ab)=\tau(ba)$ ($a,b\in \mathscr{A}$). It follows from Theorem \ref{commutator} that a linear functional $\tau$ on $\mathscr{B}(\mathcal{J}_p(\omega_1))$ is a trace if and only if it satisfies the equation $\tau=\tau(I)\Lambda_p$. Thus, $\Lambda_p$ is the only normalised trace on $\mathscr{B}(\mathcal{J}_p(\omega_1))$. Laustsen proved that the $K_0$-group of $\mathscr{B}(\mathcal{J}_p)$ is isomorphic to the additive group of integers \cite[Theorem 4.6]{laustsenktheory}; for the definition of the $K_0$-group consult e.g.~\cite{blackadar}. Laustsen's proof relies on the fact that $\mathscr{B}(\mathcal{J}_p)$ has a unique normalised trace with kernel being an ideal of operators factoring through a Banach space isomorphic to its $\ell_p$-sum. It turns out that his argument carries over to $\mathscr{B}(\mathcal{J}_p(\omega_1))$, thus we obtain as a by-product the following fact:
\begin{proposition}$K_0(\mathscr{B}(\mathcal{J}_p(\omega_1))\cong \mathbb{Z}$.\end{proposition}


\section{Weakly compactly generated operators on $C(K)$-spaces}
The aim of this section is to give some natural conditions which would guarantee that a~given operator on a~$C(K)$ space is weakly compactly generated. It is well-known that every such operator $T\colon C(K)\to X$ has a~Riesz-type representation (cf. \cite[Chapter 6]{diestel_uhl}). Namely, there exists a~weak$^\ast$-countably additive vector measure $\mu\colon\Sigma\to X^{\ast\ast}$ (called the \emph{representing measure} for $T$), defined on the $\sigma$-algebra $\Sigma$ of all Borel subsets of $K$, such that:
\begin{itemize}
\item[(i)] for each $x^\ast\in X^\ast$ the map $\Sigma\ni A\to \langle \mu(A), x^\ast\rangle$ is a~regular countably additive scalar measure (and will be denoted by $x^\ast\circ\mu$);
\item[(ii)] $\langle x^\ast, Tf\rangle=\int_K f\,{\rm d}(x^\ast\circ\mu)$ for each $x^\ast\in X^\ast$ and $f\in C(K)$;
\item[(iii)] $\n{T}=\n{\mu}(K)$.
\end{itemize}

The representing measure $\mu$ of $T$ may be defined explicitly by \[\mu(A)=T^{\ast\ast}\mathbf{1}_A\;\;\;(A\in \Sigma).\] Equivalently, 
\[\langle \mu(A), x^\ast\rangle =\mu_{x^\ast}(A),\] 
where $\mu_{x^\ast}$ is the scalar measure produced by the Riesz theorem applied for the functional $x^\ast T$.

Suppose $K$ is an \emph{Eberlein compact space}, that is, $K$ is homeomorphic to a weakly compact subset of some Banach space (we refer to \cite{LI} for an exposition concerning the class of Eberlein compact spaces), and consider the identity operator $I_{C(K)}$, which is weakly compactly generated. Then for every $A\in\Sigma$ we have $\mu(A)=\p_A$ and a~straightforward calculation gives $\n{\mu}(A)=1$, provided that $A\not=\varnothing$. So, for a~$\mathsf{WCG}$ operator on $C(K)$ it may happen that there are no non-empty sets of small semivariation, nonetheless in our example the whole domain is Eberlein. In this spirit we will prove the following result.
\begin{theorem}\label{WCG_C(K)}
Let $K$ be a compact Hausdorff space, $X$ be a~Banach space, and $T\colon C(K)\to X$ be a~bounded operator. Suppose $\mu\colon\Sigma\to X^{\ast\ast}$ is the representing measure for $T$ and for each $\e>0$ there exists a~decomposition $K=K_E^\e\cup L^\e$, where $K_E^\e$ is an Eberlein compactum and $\n{\mu}(L^\e)<\e$. Then the range of $T$ lies in a~$\mathsf{WCG}$ Banach space.
\end{theorem}

We shall make use of the characterisation of subspaces of $\mathsf{WCG}$ Banach spaces, obtained by Fabian, Montesinos and Zizler \cite{FMZ} (cf. also \cite[Theorem 6.13]{hajek}). Recall that a~subset $M$ of a~Banach space $X$ is called $\e$-{\it weakly compact} if it is bounded and $\wsc{M}\subseteq X+\e B_{X^{\ast\ast}}$.
\begin{theorem}\label{FMZ_theorem}
A Banach space is a~subspace of a~$\mathsf{WCG}$ Banach space if and only if for every $\e>0$ its unit ball can be covered by countably many $\e$-weakly compact sets.
\end{theorem}

\begin{proof}[Proof of Theorem \ref{WCG_C(K)}]
Fix $\e>0$ and define $\hat{T}\colon C(K_E^\e)\to X$ to be the unique operator whose representing measure is equal to the restriction of $\mu$ to Borel subsets of $K_E^\e$ (see \cite[Theorem VI.1.1]{diestel_uhl}). Since $K_E^\e$ is Eberlein, the range of $\hat{T}$ is weakly compactly generated. Let $M\subseteq X$ be a~convex, symmetric and weakly compact set such that $$\oo{\hat{T}(C(K_E^\e))}\subset\bigcup_{n=1}^\infty\bigl(nM+\e B_X\bigr).$$

For every $f\in C(K)$ and $x^\ast\in X^\ast$, $\n{x^\ast}\leqslant 1$, we have $$\bigl|x^\ast\bigl(Tf-\hat{T}f|_{K_E^\e}\bigr)\bigl|=\Bigl|\int_{L^\e}fd(x^\ast\circ\mu)\Bigr|\leqslant\abs{x^\ast\circ\mu}(L^\e)\n{f}\leqslant\e\n{f}.$$This implies that $$T(B_{C(K)})\subset\bigcup_{n=1}^\infty (nM+2\e B_X)$$and, since $\wsc{nM+2\e B_X}\subseteq X+2\e B_{X^{\ast\ast}}$, each of the sets $nM+2\e B_X$ is $2\e$-weakly compact. Repeating this argument for $\e/k$ instead of $\e$ (for all $k\in\N$) we get a~similar covering of $T(kB_{C(K)})$ and, consequently, for some sequence $(M_k)_{k=1}^\infty$ of weakly compact sets we have $T(C(K))\subset\ccup_{k=1}^\infty (M_k+2\e B_X)$, thus $\oo{T(C(K))}$ may be covered by countably many $3\e$-weakly compact sets. It remains to appeal to Theorem~\ref{FMZ_theorem}.
\end{proof}

Now, we present two examples. The first one shows that the implication in Theorem \ref{WCG_C(K)} cannot be reversed, while the second one shows that the assumption of this theorem does not imply that there exists a~decomposition $K=K_E\cup L$ with $K$ being Eberlein and $L$ being of semivariation zero.

\begin{example}
Consider a~map $\p\colon [0,\om_1]\to [0,\om_1]$ given by $$\p(\alpha)=\left\{\begin{array}{ll}
\alpha+1, & \mbox{if }\alpha<\om_1\mbox{ is a successor ordinal},\\
\alpha, & \mbox{if }\alpha\leqslant\om_1\mbox{ is a limit ordinal.}
\end{array}\right.$$
This is a continuous function, whence the composition operator $C_\p\colon C[0,\om_1]\to C[0,\om_1]$ defined by $C_\p f=f\circ\p$ is bounded. Now, put $T=I_{C[0,\om_1]}-C_\p$. 

Observe that $T$ maps the Schauder basis $(\ind_{[0,\alpha]})_{0\leqslant\alpha\leqslant\om_1}$ of $C[0,\om_1]$ onto the set $\{e_\alpha\}_{\alpha\in D}\cup\{0\}$, where $D$ is the set consisting of zero and all successors less than $\om_1$. Hence, the range of $T$ is isometric to $c_0(\om_1)$ which is a~$\mathsf{WCG}$ Banach space. However, as we shall see, there is no decomposition $[0,\om_1]=K\cup L$ with $K$ being Eberlein and with semivariation of $L$ less than $1$.

Let $x^\ast=(x_\alpha)_{0\leqslant\alpha\leqslant\om_1}\in\ell_1[0,\om_1]$ (identified with the dual space of $C[0,\om_1]$). For each $f\in C[0,\om_1]$ we have $$x^\ast T(f)=\sum_{\alpha\in D}x_\alpha(f(\alpha)-f(\alpha+1)),$$whence the representing measure for the functional $x^\ast T$ is given by $$\mu_{x^\ast}(\{\alpha\})=\left\{\begin{array}{ll}
x_0, & \mbox{if }\alpha=0,\\
x_\alpha-x_{\alpha^\prime}, & \mbox{if }\alpha=\alpha^\prime+1\in D,\\
0, & \mbox{otherwise}.
\end{array}\right.
$$
Let $\mu$ stand for the representing measure for $T$. By the relation $\mu(A)x^\ast=\mu_{x^\ast}(A)$, for any Borel set $A\subseteq [0,\om_1]$, and any $x^\ast\in C[0,\om_1]^\ast$, we have 
$$\abs{x^\ast\circ\mu}(A)=\sup_\pi\sum_{E_j\in\pi}\abs{\mu(E_j)x^\ast}=\sum_{\alpha\in A}\abs{\mu_{x^\ast}\{\alpha\}},$$
where $\pi$ is the set of all finite Borel partitions of $A$. Hence, whenever $A\cap D\not=\varnothing$, we have $\n{\mu}(A)\geqslant 1$.This shows that any decomposition $[0,\om_1]=K\cup L$ with $\n{\mu}(L)<1$ would imply $D\subseteq K$, thus $K$ would be homeomorphic to the non-Eberlein space $[0,\om_1]$.
\end{example}

\begin{example}
Define an operator $T\colon C(\beta\N)\cong\ell_\infty\to c_0$ by $T(\xi)=(\frac{1}{n}\xi_n)_{n=1}^\infty$. Then for each $x^\ast=(\eta_n)_{n=1}^\infty\in\ell_1$ the representing measure $\mu_{x^\ast}$ for $x^\ast T$ is supported on the set $\N$ and for each $n\in\N$ it takes the value $\frac{1}{n}\eta_n$. Then, similarly as above, we get $\abs{x^\ast\circ\mu}(A)=\sum_{n\in A\cap\N}\frac{1}{n}\abs{\eta_n}$ for every Borel set $A\subset\beta\N$. Therefore, $$\n{\mu}(A)=\frac{1}{\min(A\cap\N)}\quad\mbox{(with the convention }\frac{1}{\infty}=0\mbox{)}.$$

Consequently, in order to have $\beta\N=K_E^\e\cup L^\e$ with $\n{\mu}(L^\e)<\e$, one should only guarantee that $\min(L^\e\cap\N)>\e^{-1}$. However, if we wish that $\beta\N=K\cup L$, where $K$ is Eberlein and $\n{\mu}(L)=0$, then necessarily $\N\subseteq K$, whence $K=\beta\N$ which is not Eberlein. 
\end{example}

\section{$\WCG(C(K))$ inside $\B(C(K))$}
Let $K$ be a compact Hausdorff space. The `magnitude' of $\mathscr{W\!C\!G}(C(K))$ in $\mathscr{B}(C(K))$ may be used as a na\"{i}ve measure of the `similarity of the space $K$ to an Eberlein compactum'. The first-named author together with Koszmider and Laustsen \cite{kaniakoszmiderlaustsen} proved that for $C[0,\omega_1]$ the ideal of weakly compactly generated operators is as big as possible, that is, it has codimension one in $\mathscr{B}(C[0,\omega_1])$. On the other hand, Laustsen and Loy \cite[p.~253]{laustsenloy} noticed, based on well-known facts, that the ideal of weakly compact operators on $\ell_\infty= C(\beta \mathbb{N})$ is the unique maximal ideal of $\mathscr{B}(\ell_\infty)$. The space $\ell_\infty$ is Grothendieck \cite[Theorem VII.15]{diestelseq}, hence it follows from \cite[Corollary 5, p.~150]{Di} that $\mathscr{W}(\ell_\infty)=\mathscr{W\!C\!G}(\ell_\infty)$. However, the codimension of $\mathscr{W\!C\!G}(\ell_\infty)$ in $\mathscr{B}(\ell_\infty)$ is infinite, as can be easily seen.

We shall add to this picture a $C(K)$-space constructed by Koszmider \cite{Ko} for which the ideal of weakly compactly generated operators has codimesion one in $\mathscr{B}(C(K))$. This space has an interesting feature: we are able to give a complete description of the lattice of closed ideals of $\mathscr{B}(C(K))$ what we shall do. The first-named author was informed by P.A.H. Brooker that he also obtained a similar result (unpublished) independently.

\begin{theorem}[Koszmider \cite{Ko}]\label{Koszmider}Assuming the continuum hypothesis $\mathsf{CH}$ or Martin's Axiom with $\neg \mathsf{CH}$, there exists a compact scattered Hausdorff space $K$ such that:
\begin{romanenumerate}
\item the ideal $\mathscr{X}(C(K))$ has codimension one in $\mathscr{B}(C(K))$;
\item each separable subspace of $C(K)$ is contained in a subspace isomorphic to $c_0$;
\item if $C(K)=A\oplus B$ is a decomposition into two closed, infinite-dimensional subspaces $A$ and $B$, then either $A\cong c_0$ and $B\cong C(K)$ or vice versa\end{romanenumerate} \end{theorem}

\begin{remark}\label{mrowkaremark}The above-mentioned space $K$ is a special example of a \emph{Mr\'{o}wka space}, that is, the Stone space of the Boolean subalgebra of $\mathscr{P}(\mathbb{N})$ generated by an uncountable family of almost disjoint sets together with all finite subsets of $\mathbb{N}$. Mr\'{o}wka spaces are classical examples of scattered compacta which are not Eberlein. Consequently, the ideal $\mathscr{W\!C\!G}(C(K))$ is properly contained in $\mathscr{B}(C(K))$. Since $\mathscr{W\!C\!G}(C(K))$ contains the maximal ideal $\mathscr{X}(C(K))$ and is proper, we have $\mathscr{W\!C\!G}(C(K))=\mathscr{X}(C(K))$.\end{remark}

For this particular space $K$ we can describe the lattice of all closed ideal in $\mathscr{B}(C(K))$. To do that we need to gather some well-known facts. Each Radon measure on a compact scattered space $K$ is countably supported, whence there is a natural isometric identification between the dual space of $C(K)$ and the Banach space $\ell_1(K)$. In particular, the dual space $C(K)^*$ enjoys the Schur property. Moreover, we shall require the following theorem due to Lotz, Peck and Porta \cite{Lo}:

\begin{theorem}\label{lotz}Let $K$ be a compact Hausdorff space. Then, $K$ is scattered if and only if each closed, infinite-dimensional subspace of $C(K)$ contains a subspace which is isomorphic to $c_0$ and complemented in $C(K)$.\end{theorem}

\begin{proposition}\label{scatt}Let $K$ be a compact scattered Hausdorff space. Then,
\begin{romanenumerate}
\item $\mathscr{K}(C(K))=\mathscr{W}(C(K))$;
\item\label{scattii} no closed ideal lies between $\{0\}$ and $\mathscr{K}(C(K))$ or $\mathscr{K}(C(K))$ and $\mathscr{G}_{c_0}(C(L))$.\end{romanenumerate}\end{proposition}
\begin{proof}Part (i) is standard: the dual space $C(K)^*=\ell_1(K)$ has the Schur property, whence $\mathscr{K}(\ell_1(K))=\mathscr{W}(\ell_1(K))$. The claim follows from Gantmacher's theorem and Schauder's theorem.

To prove part (ii) let us notice that the space $C(K)$, being an $\mathscr{L}_\infty$-space, has the bounded approximation property, hence the ideal of compact operators $\mathscr{K}(C(K))$ is the smallest closed non-trivial ideal in $\mathscr{B}(C(K))$. It remains to show that if $T\in \mathscr{B}(C(K))\setminus \mathscr{W}(C(K))$ then the ideal generated by $T$ contains the ideal $\mathscr{G}_{c_0}(C(K))$. By a result of Pe\l czy\'{n}ski (cf. \cite{pelczynski} or \cite[Theorem VI.2.15]{diestel_uhl}), there is a subspace $Y$ of $C(K)$ isomorphic to $c_0$ such that $T|_Y$ is bounded below. Hence $T(Y)$ is closed, so by Theorem \ref{lotz}, it contains a completemented copy of $c_0$, say $Z$. Note that $T|_X\colon X\to Z$ is an isomorphism which factors the identity operator on $c_0$. Consequently, $\mathscr{G}_{c_0}(C(K))$ is contained in the ideal generated by $T$.\end{proof}

\begin{theorem}\label{mrowka}Let $K$ be the Mr\'{o}wka space constructed by Koszmider in \cite{Ko}. Then the the lattice of closed ideals in $\mathscr{B}(C(K))$ has the form:
\[\{0\}\subsetneq \mathscr{K}(C(K))\subsetneq \mathscr{X}(C(K))=\mathscr{G}_{c_0}(C(K))= \mathscr{W\!C\!G}(C(K))\subsetneq \mathscr{B}(C(K)).\]
\end{theorem}

\begin{proof}[Proof of Theorem \ref{mrowka}]Since the space $K$ is scattered, no closed ideal lies neither between $\{0\}$ and $\mathscr{K}(C(K))$ nor $\mathscr{K}(C(K))$ and $\mathscr{G}_{c_0}(C(K))$ (Proposition \ref{scatt}(\ref{scattii})). By Remark \ref{mrowkaremark}, we have $\mathscr{W\!C\!G}(C(K))=\mathscr{X}(C(K))$; let $T\in \mathscr{X}(C(K))$. By Theorem \ref{Koszmider}(b), there is a subspace $Y$ isomorphic to $c_0$ such that $T(C(K))\subseteq Y$, so $T$ factors through $c_0$. Thus, the ideals $\mathscr{G}_{c_0}(C(K))$ and $\mathscr{X}(C(K))$ are equal by virtue of the maximality of the latter one. Now, if $\mathscr{J}$ is any maximal ideal in $\mathscr{B}(C(K))$, by Proposition \ref{scatt}(\ref{scattii}), it must contain $\mathscr{G}_{c_0}(C(K))$, hence $\mathscr{J}=\mathscr{G}_{c_0}(C(K))$, thus there are no other closed ideals in $\mathscr{B}(C(K))$ than those listed in the claim.\end{proof}
\bibliographystyle{amsplain}

\end{document}